\documentclass{amsart} \usepackage{color} \usepackage{amssymb}
\usepackage{xypic}
\newcommand{\acom}[1]{} \newcommand{\coker}{\operatorname{coker}}
 \newcommand{\DR}{\operatorname{DR}}
\newcommand{\Dst}{\operatorname{D}_{\textup{st}}}
\newcommand{\red}{\operatorname{red}} \newcommand{\Xgrph}{\Gamma(X)}
\newcommand{\PP}{\mathbb{P}} \newcommand{\fp}{\textup{fp}}

\newcommand{\Cone}{\operatorname{Cone}}
\newcommand{\ext}{\operatorname{Ext}} \newcommand{\nt}{\tilde{N}}
\newcommand{\ntl}{\nt_l} \newcommand{\ms}{\operatorname{ms}}
 \newcommand{\Gr}{\operatorname{Gr}}
\newcommand{\val}{\operatorname{val}}
\newcommand{\res}{\operatorname{res}} \newcommand{\Zl}{\Z_l}
\newcommand{\Ql}{\Q_l}
\newcommand{\et}{\text{\'et}} \newcommand{\het}{H_{\et}}
\newcommand{\hst}{H_{\textup{st}}}
\newcommand{\cst}{C_{\textup{st}}^\bullet}
\newcommand{\cstp}{C_{\textup{st}}^{\bullet\prime}}
\newcommand{\csti}{C_{\textup{st}}^{\bullet i}}
\newcommand{\csto}{C_{\textup{st}}^{\bullet 0}}
\def\hT(#1,#2,#3){H_{\mathcal{T}}^{#2}(#1,\Z(#3))}
\def\hmot(#1,#2,#3){H_{\mathcal{M}}^{#2}(#1,\Z(#3))}
\def\hf(#1,#2,#3,#4){H_{\fp,#1}^{#2}(#3,#4)}
\def\hft(#1,#2,#3,#4){\tilde{H}_{\fp,#1}^{#2}(#3,#4)}
\def\hfp(#1,#2,#3){H_{\fp}^{#1}(#2,#3)}
\def\hsyn(#1,#2,#3){H_{\syn}^{#1}(#2,#3)}
\def\hmsyn(#1,#2,#3){\tilde{H}_{\ms}^{#1}(#2,#3)}
\def\hfpt(#1,#2,#3){\tilde{H}_{\fp}^{#1}(#2,#3)}
\def\hfP(#1,#2,#3,#4){H_{\textup{f},#1}^{#2}(#3,#4)}
\def\hfPt(#1,#2,#3,#4){\tilde{H}_{\textup{f},#1}^{#2}(#3,#4)}
\def\bdr(#1,#2){\frac{#1_g #2_h}{1-\alpha #1_g #2_h}}
 
\newcommand{\reg}{\operatorname{reg}} \newcommand{\regt}{\reg_t}
\newcommand{\regtl}{\reg_t^l} 
\newcommand{\Q}{\mathbb{Q}} \newcommand{\XX}{\mathcal{X}}

\renewcommand{\O}{\mathcal{O}} \newcommand{\dlog}{d\log}
\newcommand{\Z}{\mathbb{Z}} \newcommand{\Qp}{\Q_p}
\newcommand{\Zp}{\Z_p} \newcommand{\Cp}{\mathbb{C}_p}
 \newcommand{\Kbar}{\bar{K}}
\newcommand{\pair}[1]{{\left\langle #1 \right\rangle}}
\newcommand{\dr}{\textup{dR}} \newcommand{\hdr}{H_{\dr}}

 \newcommand{\syn}{\textup{syn}}
 
\newcommand{\isom}{\cong}
\newcommand{\tors}{/\textup{tor}}
\newcommand{\Ymm}[1]{\bar{Y}^{(#1)}}
\newcommand{\Ym}{\Ymm{m}}
\newcommand{\cijk}{C_j^{i,k}}
\newcommand{\cij}{C_j^{i}}
\newcommand{\Cijk}{CH^{i+j-k}(\Ymm{2k-i+1})}
\newcommand{\ct}[3]{C_{#1}^{#2,#3}}
\newcommand{\ctt}[2]{C_{#1}^{#2}}

\newcommand{\Cy}[2]{CH^{#1}(\Ymm{#2})}
 \newtheorem{conjecture}{Conjecture}
\newtheorem{theorem}{Theorem}[section]

\newtheorem{proposition}[theorem]{Proposition}
\newtheorem{lemma}[theorem]{Lemma}
\newtheorem{corollary}[theorem]{Corollary} \theoremstyle{definition}
\newtheorem{definition}[theorem]{Definition}
\newtheorem{remark}[theorem]{Remark}

\newtheorem{assumption}[theorem]{Assumption}
\numberwithin{equation}{section}
\begin{document}
\title[Toric regulators]{Toric regulators} \author{Amnon 
Besser}
\address{
  Department of Mathematics\\
  Ben-Gurion University of the Negev\\
  P.O.B. 653\\
  Be'er-Sheva 84105\\
  Israel } \email{bessera@math.bgu.ac.il}

\author{Wayne Raskind} \address{
  Department of Mathematics\\
  Wayne State University\\
  Detroit, MI 48202\\
  U.S.A.  } \email{raskind@wayne.edu}
\maketitle{}

\section*{Introduction}
\label{sec:intro}
In the mid-19th century, Dirichlet (for quadratic fields) and then Dedekind defined a regulator map relating the units in the ring of integers of an algebaic number field of finite degree over $\Q$ with $r_1$ real embeddings and $2r_2$ complex embeddings to a lattice of codimension one in a Euclidean space of dimension $r_1+r_2$.  They then showed how a determinant formed from this map and other invariants of the field are related to values of zeta and $L$-functions, known as Dirichlet's class number formula (\cite[Supplemente, V, $\S\S$ 183 and 184]{DirDed94}).  Since then, the term ``regulator'' has been applied to many such maps in number theory and algebraic geometry such as higher algebraic $K$-theory of number fields, Abel-Jacobi maps for algebraic cycles, and more generally, motivic cohomology.   In most cases, the source of the regulator map is a group of interest that is deemed to be difficult to compute, and the target somewhat easier to compute.  A very general form of this circle of ideas can be found in Beilinson's conjectures relating motivic cohomology of a smooth projective variety over a number field to real Deligne cohomology \cite{Bei84} and values of $L$-functions, and their refinement by Bloch-Kato \cite{Blo-Kat90}.  There are $p$-adic analogues of these conjectures, where real Deligne cohomology is replaced by a suitable $p$-adic cohomology theory such as syntomic or log-syntomic cohomology, and a conjectural relationship with values of $p$-adic $L$-functions.  In the special but important case of a variety with totally degenerate reduction over a $p$-adic field $K$ (please see below for definitions), this paper seeks to tie many of the $p$-adic conjectures and some of the known results together under the guise of what we call \emph{toric regulators}, which relate motivic cohomology with $p$-adic tori (quotient of a multiplicative torus by a finitely generated free abelian group).  These tori may be compact or not.

In~\cite{Ras-Xar07,Ras-Xar07a}, the second named author and
Xarles studied a class of varieties $X$ over a local field $K$ with
what they termed \emph{totally degenerate
reduction}. In~\cite{Ras-Xar07} they studied the \'etale cohomology
of $X$ with $\Z_l$-coefficients and showed that for \emph{all} $l$ these are, up
to finite torsion and cotorsion, extensions of direct sums of Tate
twists. In~\cite{Ras-Xar07a} they used this result to define \emph{$p$-adic
intermediate Jacobians}, which are $p$-adically uniformized tori,
together with Abel-Jacobi maps from the Chow group of algebraic cycles that are homologically equivalent to zero.  The first example of these is the Tate elliptic curve
$E_q$, which is given rigid analytically by $\mathbb{G}_m/q^{\Z}$ with $q$ of absolute value less than 1 in $K$. In
this case, the intermediate Jacobian is just $K^\times/ q^{\Z}$, and
the Abel-Jacobi map is essentially the identity. More generally, for a
$p$-adically uniformized curve $X$, their work recovers the $p$-adic
uniformization of the Jacobian in a purely algebro-geometric way.  The Abel-Jacobi map defined in loc. cit. should agree with that provided by
Manin-Drinfeld~\cite{ManDri73}, although they do not prove that in their paper.  

The work~\cite{Ras-Xar07,Ras-Xar07a} raises the very natural question
of defining toric regulators in higher motivic cohomology, and that is the main purpose of this paper. We will define
\emph{higher intermediate Jacobians} $ \hT(X,k+1,r)$, given by the quotient
of an algebraic torus by periods, and construct, assuming a certain natural conjecture,
regulator maps into them, \emph{toric regulators}.

The toric regulator is a refinement of the regulator of
Sreekantan~\cite{Sre10}, which is a map
\begin{equation*}
  r_{\mathcal{D}}: \hmot(X,k+1,r) \to H_{\mathcal{D}}^{k+1}(X,\mathbb{Q}(r))
\end{equation*}
where the group on the right is the cohomology of a certain cone
defined by Consani~\cite{Con98}.  It is a finite dimensional $\Q$-vector space.  We construct a valuation map
(see~\eqref{eq:valmap})
\begin{equation*}
  \hT(X,k+1,r) \to  H_{\mathcal{D}}^{k+1}(X,\Z(r))
\end{equation*}
and
conjecture that after tensoring with $\Q$, the Sreekantan regulator is the valuation of the toric
regulator.

Another interesting feature of the theory of the toric regulator is
the relation with the syntomic regulator. Following the case of $CH^1$
of curves, where the toric regulator is the Abel-Jacobi map and the
syntomic regulator is its logarithm, one expects that ``the
syntomic regulator is the logarithm of the toric regulator.'' We
formulate this assertion precisely and prove it.  From the point of
view of the syntomic theory, this adds the interesting assertion that
the syntomic regulator can be ``exponentiated.'' This sometimes allows
us to guess formulas for the toric regulator, and we will describe one
such guess, but without presenting the syntomic motivation, for
brevity.

The Tate elliptic curve is in some sense the original toric regulator.  The paper ~\cite{Ras-Xar07,Ras-Xar07a} may be viewed as a purely algebro-geometric way using $p$-adic Hodge theory to interpret and generalize Tate's analytic theory and its generalization to curves of higher genus and abelian varieties by Mumford~\cite{Mum72a,Mum72b}. 
Another example is
provided by $K_2$ of a curve $X$ with totally degenerate reduction,
i.e., a Mumford curve. In this case, it turns out a regulator into an
algebraic torus has already been developed, and termed the \emph{rigid
analytic regulator} by P\'al~\cite{Pal10}
(we prove this except at the prime $p$).  We also compare the log of the rigid analytic regulator with the
syntomic regulator, as computed in~\cite{Bes18}. For the product of
two Mumford curves, we explain a conjectural formula for the toric
regulator, whose motivation is syntomic.

A question left for future work is the relation between the toric
regulator and $L$-functions. Because the syntomic regulator is the
logarithm of the toric regulator and is related to special values of
$p$-adic $L$-functions, we are looking for such special values that may
be ``exponentiated.'' There have been several instances of such a
phenomenon, starting with the refined Birch and Swinnerton-Dyer
conjecture of Mazur and Tate~\cite{Maz-Tat87} and its descendents,
especially in Darmon's work on Stark-Heegner
points~\cite{BerDar94,Ber-Dar96,Ber-Dar98,Dar98,Dar01}. These
conjectures concern rational points on elliptic curves in terms of the
Tate parameterization at a prime of split multiplicative reduction, and
so fit perfectly with the toric regulator in this case. These
conjectures inspired in turn the refined $p$-adic Stark conjecture of
Gross~\cite{Gross88}. There is one example in higher
K-theory, which is due to P{\'a}l~\cite{Pal10b} and Kondo and
Yasuda~\cite{KonYas11}, providing an $L$-function and a regulator
formula in the case of $K_2$ of the the Drinfeld modular curve, in
analogy with Beilinson's work in the classical case~\cite{Bei84}
and~\cite{Ber-Dar12,Bru10,Nik10} in the $p$-adic case on $K_2$ of a
modular curve. Note that to relate this with
the toric regulator, one has to either import the result to the number
field case or extend the toric regulator to the function field case
(the theory of~\cite{Ras-Xar07} at the prime $p$ currently relies on
$p$-adic Hodge theory).

Another possible source of examples is the Sreekantan regulator. This
is expected to be connected with $L$-values~\cite{Sre10}, and the example
of $K_1$ of the product of two Drinfeld modular curves has been worked
out by Sreekantan~\cite{Sre10a}. Because of the relation between the
toric and Sreekantan regulators mentioned above, the sought after $L$-functions should be
such that their valuation is the corresponding $L$-function of
Sreekantan.

This work began while the first author was on sabbatical at Arizona State
University, continued while the second author
visited the first at Oxford University, and then during two visits of the first author to Wayne
State University.  It was
completed while the first author was on sabbatical at the Georgia Institute of
Technology and then a member of IH\'ES. We thank all of these institutions
and the Raymond and Beverly Sackler Foundation, whose
fellowship supported the stay at IH\'ES. The first author is currently
supported by a grant from the Israel Science Foundation number 912/18.

\section{The toric regulator}\label{sec:toric-regulator}

Let $K$ be a finite extension of $\Qp$ with ring of integers $R$,
uniformizer $\pi$ and residue field $F$. Let $G$ be the absolute
Galois group $G=\operatorname{Gal}(\bar{K}/K)$. Let $X$ be a
smooth, projective, geometrically connected variety over $K$ which has
totally degenerate reduction in the sense of~\cite{Ras-Xar07}.  Let
$k$ and $r$ be non-negative integers such that
\begin{equation}
  \label{eq:numerics}
  k-2r\le -1
\end{equation}
In this section we define the toric regulator
\begin{equation*}
  \hmot(X,k+1,r)_0 \xrightarrow{\regt} \hT(X,k+1,r)
\end{equation*}
from the motivic cohomology of $X$ to the toric higher intermediate
Jacobian of $X$. The subindex $0$ on motivic cohomology refers to
homologically trivial classes. This is only relevant when $k+1=2r$, in
which case $ \hmot(X,k+1,r) = CH^r(X)$ are Chow groups. In this
particular case the toric regulator is constructed in detail
in~\cite{Ras-Xar07a}. We will therefore mostly concentrate on the case
of strong inequality in~\eqref{eq:numerics}.

Let $l$ be a prime number.  As a consequence of the construction of
\'etale realization functors~\cite{Ivo07}, one gets an \'etale
regulator map (a rational coefficients version was known for a long
time but we will need integral coefficients),
\begin{equation}\label{regl}
  \reg_l: \hmot(X,k+1,r)_0 \to H^1(K,M_l(r))
  \text{ with } M_l=H_{\et}^{k}(X\otimes_K \bar{K}, \Z_l)
\end{equation}
(see for example~\cite{Rio06}).  Note that we have the
restriction~\eqref{eq:numerics} as otherwise the motivic cohomology
groups vanish.

As $X$ has totally degenerate reduction we have the following result
of Raskind and Xarles (see~\cite[Cor. 1 and Theorem 3]{Ras-Xar07} as
well as~\cite[Theorem 3]{Ras-Xar07a} summarizing both the $l\ne p$ and
$l=p$
cases).
\begin{theorem}
  There exist finitely generated abelian groups $T^i_j $ and, for each
  $l$, a filtration $W_\bullet$ on the $\Z_l$-module $M_l$ and
  isogenies
  \begin{equation*}
    \Gr_{i}^W M_l \tors \to
    \begin{cases}
      T_{\frac{k+i}{2}}^{-i}\tors \otimes \Zl({\frac{-k-i}{2}}) & k+i \text{ even}\\
      0 & \text{otherwise,}
    \end{cases}
  \end{equation*}
  which are isomorphisms for almost all $l$.
\end{theorem}

Let us recall the construction of the groups $T^{i}_{j}$~\cite[Section
3]{Ras-Xar07}. By assumption, $X$ is the generic fiber of a proper
$\O_K$-scheme $\mathcal{X}$ with strictly semi-stable reduction and special
fiber $Y$, which decomposes as a union $Y= \cup_{i=1}^n Y_i$. Set, for
each subset of indices $I\subset \{1,\ldots,n\}$,
\begin{equation*}
  Y_I = \bigcap_{i\in I} Y_i\;.
\end{equation*}
Let $\bar{Y}_I:= Y_I\otimes \bar{F}$ and let $\Ym$ be the disjoint
union of $\bar{Y}_I$ over all subset $I$ of size $m$.  We first define
groups $\cijk= \Cijk$ for each triple $(i,j,k)$ such that
$k\ge \max(0,i)$ and $0$ otherwise. Then we set
\begin{equation*}
  \cij = \bigoplus_k \cijk \;.
\end{equation*}
To make a complex out of these groups, we define the following maps:
For a subset of indices $I$ of size $m+1$ and an integer $0<r\le m+1$,
we define $I_r$ to be the subset obtained from $I$ by deleting the
$r$'th index. There is an obvious inclusion map
$\rho_r: Y_I \to Y_{I_r} $ and we define maps,
\begin{align*}
  \theta_{i,m} &= \sum_{r=1}^{m+1} (-1)^{r-1} \rho_r^\ast: CH^i(\Ym) \to CH^i(\Ymm{m+1})\;,\\
  \delta_{i,m} &= \sum_{r=1}^{m+1} (-1)^{r} \rho_{r\ast}:  CH^i(\Ymm{m+1}) \to CH^{i+1}(\Ym)\;,\\
  d' &= \bigoplus_{k\ge \max(0,i)} \theta_{i+j-k,2k-i+1}\;, \\
  d'' &= \bigoplus_{k\ge \max(0,i)} \theta_{i+j-k,2k-i}\;,
\end{align*}
and finally \[d_j^i = d'+d'': \cij \to C_j^{i+1}\;.\] Then we define
\begin{equation*}
  T_j^i := H^i(C_j^\bullet)\;.
\end{equation*}
The monodromy map
\begin{equation}\label{monodromy}
  N: T_j^i \to T_{j-1}^{i+2}
\end{equation}
is induced by the map $\cij \to C_{j-1}^{i+2}$ which is the identity
on the common factors. The composed map
\begin{equation}
  N^i: T_{j+i}^{-i} \to T_{j}^{i}
\end{equation}
is an isogeny for $i\ge 0$~\cite[Proposition 1]{Ras-Xar07}, implying
that $N$ is injective for negative $i$ and surjective for positive $i$
after tensoring with $\Q$.
\begin{remark}\label{inj}
  The numerical condition~\eqref{eq:numerics} on $k$ and $r$ imply
  that $T_{r}^{k-2r} \to T_{r-1}^{k+2-2r} $ is injective (after
  tensoring with $\Q$) while $ T_{r}^{k+1-2r} \to T_{r-1}^{k+3-2r} $
  is injective except when $k-2r=-1$ in~\eqref{eq:numerics},
  i.e., the case of cycles.
\end{remark}
There exists a pairing~\cite[p. 274]{Ras-Xar07}
\begin{equation}
  \label{eq:pairing}
  (~,~): T_j^i \times T_{d-j}^{-i} \to \Z
\end{equation}
coming from the intersection pairing and inducing a duality on the
torsion free quotients~\cite[Proposition 1]{Ras-Xar07}. We further
have the relation
\begin{equation}
  \label{eq:Nselfd}
  (Nx,y)=-(x,Ny)\;.
\end{equation}
Let us now make the following simplifying assumption.
\begin{assumption}
  The Galois group of the residue field $F$ acts trivially on all the
  groups $\cijk$.
\end{assumption}
This can always be achieved after a finite field extension. It is tedious but possible to keep track of the Galois action if this were not the case, and would be important if we considered e.g. towers of field extensions of $K$.  Let us
further use the following terminology.
\begin{definition}
  We say that a map defined for each prime $l$ is an \emph{almost injection}
  (respectively, \emph{almost surjection}) if its kernel (respectively,
  cokernel) is finite for all $l$ and $0$ for almost all $l$. We say
  it is an almost isomorphism if it is both an almost injection and an
  almost surjection.
\end{definition}

To proceed with the construction of the toric regulator, we will want,
as in~\cite{Ras-Xar07a}, to isolate from $M_l(r)$ the subquotient
which is an extension of $T_{\ast}^{\ast} \otimes \Z_{l}$ by
$T_{\ast}^{\ast} \otimes \Z_{l}(1)$, which, with our indexing, is the
the subquotient $W_{2r-k} M_l(r)/W_{2r-k-4} M_l(r) $.   In order to do this, we will use the groups $H^1_g$ of Bloch-Kato (\cite{Blo-Kat90}, \S 3). Recall that when $l\ne p$ $H_g^1 = H^1 $, while when $l=p$ we have for a $\Qp$-representation $V$ that $H_g^1(K,V) = \ker: H^1(K,V) \to H^1(K,V\otimes B_{\dr}) $ and for a $\Zp$-module $g$ cohomology classes are the ones that become $g$ after tensoring with $\Q$.  The regulator $\reg_l$ from \eqref{regl} takes values in  $H_g^1(K,M_l(r))$. This is tautological for $l\ne p$ and follows from the work of  Nekov{\'a}{\v{r}} and Nizio\l~\cite{Nek-Niz14} when $l=p$ (see Section~\ref{sec:syntomic} \eqref{regp} and the following statement and~\eqref{eq:stisg}).
\begin{proposition}\label{multn}
  There exists an integer $n_0$ and well defined maps
  \begin{equation*}
    \reg_l^\prime: \hmot(X,k+1,r)_0 \to H_g^1(K,W_{2r-k} M_l(r))
  \end{equation*}
  such that for any prime $l\ne p$ and any
  $\alpha \in \hmot(X,k+1,r) $ we have
  \begin{equation*}
    \reg_l^\prime (\alpha) = n_0 x, \text{ with } x\in H_g^1(K,W_{2r-k} M_l(r))\;,\; \iota_{2r-k}(x) = n_0 \reg_l(\alpha)\;,
  \end{equation*}
  with
  \begin{equation*}
    \iota_r: W_r M_l \to M_l
  \end{equation*}
  the obvious injection.
\end{proposition}
\begin{proof}
  The quotient $M_l(r)/W_{2r-k} M_l(r)$ is, up to torsion, an iterated
  extension of copies of $\Zl(j)$ for $j<0$.   The $H^0$ of these groups is trivial, and by (\cite{Blo-Kat90} Example 3.9), we have $H^1_g(\Zl(j))=H^0(\Ql/\Zl(j))$, which are finite and killed by a fixed integer that is independent of $l$.
\end{proof}

Projecting on the quotient by $W_{2r-k-4} M_l(r) $ we obtain a map
\begin{equation}
  \label{eq:twostepreg}
  \hmot(X,k+1,r)_0 \to H_g^1(K,W_{2r-k} M_l(r)/W_{2r-k-4} M_l(r))\;.
\end{equation}
Now we can proceed in a similar manner to the proof of
Proposition~\ref{multn}. The Galois module
$W_{2r-k} M_l(r)/W_{2r-k-4} M_l(r)$ gives us an extension class
in
\[\ext^1(W_{2r-k} M_l(r)/W_{2r-k-2} M_l(r),W_{2r-k-2}
M_l(r)/W_{2r-k-4} M_l(r))\;.\] This is almost isomorphic to
$\ext^1 (T_{r}^{k-2r} \otimes \Zl, T_{r-1}^{k+2-2r} \otimes \Zl(1) )$
and so after multiplying  by an integer $n_1$ we get Galois
modules $M_l^\prime$ with a short exact sequence
\begin{equation}\label{mlpdiag}
  0\to T_{r-1}^{k+2-2r} \otimes \Zl(1) \to M_l^\prime \to  T_{r}^{k-2r} \otimes \Zl \to 0\;,
\end{equation}
and after multiplying again by an integer $n_2$ we get a regulator
map
\begin{equation}\label{regpp}
  \reg_l^{\prime\prime} : \hmot(X,k+1,r)_0 \to H_g^1(K,M_l^\prime)\;.
\end{equation}
We now consider boundary maps in the long cohomology sequence coming
from~\eqref{mlpdiag}. In degree $0$ we use Kummer theory to get the
map
\begin{equation}\label{augl}
  \ntl: T_{r}^{k-2r} \otimes \Zl \to H_g^1(K,T_{r-1}^{k+2-2r}\otimes \Zl(1) )\isom T_{r-1}^{k+2-2r} \otimes K^{\times(l)}\;,
\end{equation}
where $K^{\times(l)} $ is the $l$-completion of $K^\times$.  This is
essentially the monodromy paring considered by Raskind and
Xarles~\cite[p 6064]{Ras-Xar07a} (although they only define it in some
cases).  Let's call this the \emph{augmented monodromy} (at $l$).

Suppose now that $l\ne p$. The $l$-part of the tame inertia group is
isomorphic to $\Zl(1)$ as a Frobenius module and we identify the two
for convenience. The following is well known.
\begin{lemma}
  The following diagram commutes
  \begin{equation*}
    \xymatrix{
      K^\times \ar[rr]^{\textup{Kummer}} \ar[d]^{\val} & & H^1(K,\Zl(1))\ar[d] \\ 
      \Z \ar[rr] & & \Zl
    }
  \end{equation*}
  where the vertical map on the right is obtained by restriction to
  $\Zl(1)$.
\end{lemma}
From this Lemma and the relation between the monodoromy on \'etale
cohomology and on the $T$'s the following is easy for $l\ne p$. For $l=p$ it will be proved in Proposition~\ref{missingpiece}.
\begin{corollary}
  The map
  \[T_{r}^{k-2r} \otimes \Zl \xrightarrow{\ntl} T_{r-1}^{k+2-2r}\otimes K^{\times(l)}
  \xrightarrow{\val} T_{r-1}^{k+2-2r}\otimes
  \Zl\] is just the monodromy map~\eqref{monodromy} with the
  appropriate indexing tensored with $\Zl$.
\end{corollary}
By Remark~\ref{inj} the composed map is almost injective.
\begin{lemma}\label{keylemma}
  The obvious map from $K^\times$ to the pushout of
  \begin{equation*}
    \xymatrix{
      {\prod_l}  K^{\times(l)} \ar[r]^{\val} & \prod_l \Zl\\
      & \Z \ar[u]
    }
  \end{equation*}
  is an isomorphism. For $l\ne p$ we have the short exact sequence
  \begin{equation*}
    0\to (F^\times)_{l\textup{-torsion}} \to  K^{\times(l)} \xrightarrow{\val} \Zl \to 0\;.
  \end{equation*}
\end{lemma}
\begin{proof}  This is well-known and follows from the fact that the group of units in the ring of integers in $K$ is compact and complete with respect to its subgroups of finite index.  This is not the case for a ``larger'' nonarchimedean valued field such as $\Cp$.
\end{proof}
\begin{corollary}
  There is an \emph{augmented monodromy} map
  \begin{equation}\label{aug}
    T_{r}^{k-2r}  \xrightarrow{\nt} T_{r-1}^{k+2-2r}\otimes K^{\times}
  \end{equation}
  which gives the augmented monodromy at $l$~\eqref{augl} after
 $l$-completion for each $l$.
\end{corollary}
We can finally define one of the main objects of this paper.
\begin{definition}
  The higher toric intermediate Jacobian of $X$ in degree $k+1$ and
  twist $r$ is defined by
  \begin{equation*}
    \hT(X,k+1,r) := \coker \left(  T_{r}^{k-2r}  \xrightarrow{\nt} T_{r-1}^{k+2-2r}\otimes K^{\times} \right)\;.
  \end{equation*}
\end{definition}
\begin{proposition}
  The boundary map in the long exact cohomology sequence
  of~\eqref{mlpdiag},
  \begin{equation*}
    H_g^1(K,T_{r}^{k-2r} \otimes \Zl)  \to H^2(K,T_{r-1}^{k+2-2r} \otimes \Zl(1))\;,
  \end{equation*}
  is almost injective.
\end{proposition}
\begin{proof}
  Suppose first that $l\ne p$. By local Tate duality and by the
  duality induced by the pairing~\ref{eq:pairing} and the relation
  with the monodromy operator given in~\eqref{eq:Nselfd} we see that
  this map is almost dual to the map
  \begin{equation*}
    H^0(K,T_{d-r+1}^{2r-k-2} \otimes \Zl)  \to H^1(K,T_{d-r}^{2r-k} \otimes \Zl(1))
  \end{equation*}
  obtained from the dual of~\eqref{eq:Nselfd}, but this is again the augmented
  monodromy map, this time in the range where after applying the valuation it is almost surjective, hence is almost surjective by Lemma~\ref{keylemma}. For the case $l=p$ see the syntomic
  theory of Section~\ref{sec:syntomic}, in particular, Theorem~\ref{logissyn}.
\end{proof}
\begin{corollary}\label{nthreeneed}
  The group $ H_g^1(K,M_l^\prime) $ is almost isomorphic to
  \[\coker \left( T_{r}^{k-2r} \otimes \Zl \xrightarrow{\ntl} T_{r-1}^{k+2-2r}\otimes K^{\times(l)} 
     \right)\;. \]
\end{corollary}
\begin{definition}\label{toricl}
  The toric regulator completed at $l$ is the map $\regtl$ defined as
  the composition
  \begin{equation*}
    \hmot(X,k+1,r) \xrightarrow{ \reg_l^{\prime\prime}}  H_g^1(K,M_l^\prime) \xrightarrow{n_3} \coker \left( T_{r}^{k-2r} \otimes \Zl \xrightarrow{\ntl} T_{r-1}^{k+2-2r}\otimes K^{\times(l)}   \right)
  \end{equation*}
  of the regulator $ \reg_l^{\prime\prime} $ from~\eqref{regpp} and
  multiplication by an integer $n_3$ which is done to eliminate the
  difference between the two groups in Corollary~\ref{nthreeneed}.
\end{definition}
The main object of this work is now given by the following result.
\begin{theorem}
  Suppose conjecture~\ref{conjsreek} is true. Then there
  exists a unique map, called the \emph{toric regulator},
  \begin{equation*}
    \hmot(X,k+1,r) \xrightarrow{\regt}  \hT(X,k+1,r)\;,
  \end{equation*}
  such that for each prime $l$ the composed map
  \begin{equation*}
    \begin{split}
      \hmot(X,k+1,r) \xrightarrow{\regt} \hT(X,k+1,r) &= \coker \left(
        T_{r}^{k-2r} \xrightarrow{\nt} T_{r-1}^{k+2-2r}\otimes
        K^{\times} \right)\\ &\xrightarrow{\otimes \Z_l} \coker \left(
        T_{r}^{k-2r} \otimes \Zl \xrightarrow{\ntl} T_{r-1}^{k+2-2r}\otimes K^{\times(l)} \right)
    \end{split}
  \end{equation*}
  is the toric regulator completed at $l$ of Definition~\ref{toricl}.
\end{theorem}
\begin{proof}
  By Lemma~\ref{keylemma} it suffices to show the existence of a map
  \begin{equation*}
    r_{\mathcal{D}}  :  \hmot(X,k+1,r) \to  \coker \left( T_{r}^{k-2r} \xrightarrow{N} T_{r-1}^{k+2-2r} \right)\otimes \Q
  \end{equation*}
  such that for any prime $l$ the diagram
  \begin{equation}\label{sreekcom}
    \xymatrix{
      { \hmot(X,k+1,r) } \ar[r]^{\regtl} \ar[d]^{ r_{\mathcal{D}} } & 
      \coker \left( T_{r}^{k-2r} \otimes \Zl \xrightarrow{\ntl} T_{r-1}^{k+2-2r} \otimes  K^{\times(l)}  \right) \ar[d]^{\val} \\
      { \coker \left( T_{r}^{k-2r} \xrightarrow{N} T_{r-1}^{k+2-2r} \right)\otimes \Q} \ar[r] &
      \coker \left( T_{r}^{k-2r} \otimes \Ql \xrightarrow{N\otimes \Ql} T_{r-1}^{k+2-2r} \otimes \Ql  \right)
    }
  \end{equation}
  commutes. In Section~\ref{sreek} we will show how the work of
  Sreekantan~\cite{Sre10} gives precisely such a map
  and Conjecture~\ref{conjsreek} will say that the
  diagram~\eqref{sreekcom} commutes.
  This proved the theorem.
\end{proof}
\begin{remark}
  We hope to prove Conjecture~\ref{conjsreek} in future work.
  It is somewhat problematic that the toric regulator is only defined
  after multiplication by an integer, thereby erasing finer roots of
  unity information. Fortunately, In various situation no such
  multiplication is required. We could conjecture that in fact the
  resulting regulator has a canonical root, but we have no evidence to
  support this.
\end{remark}

\section{The relation with the regulator of Sreekantan}
\label{sreek}
In~\cite{Sre10} Sreekantan constructs a new type of regulator and
conjectures relations with special values of $L$-functions in the
function field case. We will
conjecture
that, in a very precise sense, the Sreekantan regulator is exactly the the toric regulator followed by the
valuation map. As explained in
Section~\ref{sec:toric-regulator}, this is an important step in
actually showing the existence of the toric regulator.

Let us recall the setup for Sreekantan's work. Let $X$ be smooth and
proper over $K$ with semi-stable reduction. In his setup we are not
assuming that the reduction is completely degenerate. Sreekantan
starts with a variety over a global field but this is for the sake of
getting results about L-functions and he is really only interested in
the completion at a finite prime for computing the regulator.

Sreekantan defines certain cohomology groups, for which he first
recalls work of Consani. Consani defines~\cite[(3.13)]{Con98} groups
$K^{i,j,k}$. In fact, it is easy to check (see also Observation 1 on
p. 273 of~\cite{Ras-Xar07}, noting that Consani's convention is the
same as that of~\cite{BGS97} and~\cite{GuiNav90}) that we have
\begin{equation*}
  K^{i,j,k} = C_{\frac{d+j-i}{2}}^{i,k} \otimes \Q\;.
\end{equation*}
There are differentials and a monodromy operator which are the same
as~\cite{Ras-Xar07} and Consani defines (immediately following (3.1))
\begin{equation*}
  K^{i,j} = \bigoplus_k K^{i,j,k} = C_{\frac{d+j-i}{2}}^{i} \otimes \Q \;.
\end{equation*}
We make the following definition after Consani.
\newcommand{\Cons}{\mathcal{C}}
\begin{definition}[Consani~\cite{Con98}]
  The Consani complex with twist $r$ is the complex
  \begin{equation*}
    \Cons(r) = \Cone(N: K^{*-2r,*-d} \to K^{*-2r+2,*-d})
    = \Cone(N: C_{r}^{*-2r} \to C_{r-1}^{*-2r+2} )\otimes \Q \;.
  \end{equation*}
\end{definition}
For the normalization of the cone here see~\cite[p. 331]{Con98}.
\begin{definition}
  The Deligne cohomology group $H_{\mathcal{D}}^{k+1}(X,\Q(r)) $ is
  the $k+1-2r$ cohomology of the Consani complex $\Cons(r)$.
\end{definition}
Comparing with Sreekantan the reader will observe that we have removed
the $v$-notation, which was to indicate working with the completion of
the global $X$ at the finite prime $v$.
\begin{proposition}
  Suppose now that $X$ has totally degenerate reduction. Then there is
  a long exact sequence
  \begin{equation*}
    \cdots \to T_{r}^{k-2r} \otimes \Q \to T_{r-1}^{k+2-2r}  \otimes \Q\to
    H_{\mathcal{D}}^{k+1}(X,\Q(r)) \to T_{r}^{k+1-2r} \otimes \Q \to
    T_{r-1}^{k+3-2r}  \otimes \Q\to \cdots\;.
  \end{equation*}
\end{proposition}
\begin{proof}
  This follows immediately from the definition of the Consani complex
  as a cone.
\end{proof}
\begin{proposition}
  Suppose that the $\bar{Y}_I$ are toric or cellular varieties (which is the case for all known examples of varieties with totally degenerate reduction) .  Then we have an isomorphism
  \begin{equation}\label{Consisom}
    H_{\mathcal{D}}^{k+1}(X,\Q(r)) \isom CH^{r-1}(Y,2r-k-2)\otimes \Q
  \end{equation}
  with Bloch's higher Chow groups.
\end{proposition}
\begin{proof}
  This is a combination of results in~\cite{Con98}: Lemma~3.1 (see
  also p. 341 where the lemma is used but without precise reference in
  our context) and equation~(2.3), which relies in turn on
  Conjecture~2.1, which is known for toric or cellular varieties.  Unfortunately, it is not known at present that this follows directly from the definition of totally degenerate, although it is expected to be true for any smooth projective variety over a finite field.
\end{proof}
\begin{definition}\label{sreekreg}
  The Sreekantan regulator is a map
  \begin{equation*}
    r_{\mathcal{D}}: \hmot(X,k+1,r) \to H_{\mathcal{D}}^{k+1}(X,\Q(r))
  \end{equation*}
  which is the composition of the boundary map (effectively in motivic
  homology)
  \begin{equation*}
    \hmot(X,k+1,r) \isom CH^r(X,2r-k-1)\xrightarrow{\partial}  CH^{r-1}(Y,2r-k-2)\otimes \Q\;,
  \end{equation*}
  with the isomorphism~\eqref{Consisom}.
\end{definition}

By Remark~\ref{inj}, unless $k+1=2r$, the map
$ T_{r}^{k+1-2r} \to T_{r-1}^{k+3-2r}$ is injective after tensoring
with $\Q$, giving an isomorphism
\begin{equation}\label{Tsrik}
  H_{\mathcal{D}}^{k+1}(X,\Q(r)) \isom \operatorname{coker} \left(T_{r}^{k-2r} \to T_{r-1}^{k+2-2r}\right)\otimes \Q \;.
\end{equation}
Thus, the following result is quite natural.
\begin{conjecture}\label{conjsreek}
  Assume $k+1 < 2r$. Then, with $ r_{\mathcal{D}}$ as in
  Definition~\ref{sreekreg}, diagram~\eqref{sreekcom} commutes.
\end{conjecture}
As noted in Section~\ref{sec:toric-regulator}, the existence of the
toric regulator follows from this
conjecture.  In addition, we get the following obvious relation
between the toric regulator and the Sreekantan regulator: Define the
valuation map
\begin{equation}
  \label{eq:valmap}
  \val: \hT(X,k+1,r) \to  H_{\mathcal{D}}^{k+1}(X,\Q(r))
\end{equation}
as the composition of the valuation map and the
isomorphism~\eqref{Tsrik}.
\begin{corollary}
  Assuming Conjecture~\ref{conjsreek} the following diagram commutes.
  \begin{equation*}
    \xymatrix{
      \hmot(X,k+1,r) \ar[dr]^{ r_{\mathcal{D}}} \ar[r]^{\regt} & \hT(X,k+1,r)
      \ar[d]^{\val} \\
      &
      H_{\mathcal{D}}^{k+1}(X,\Q(r))\;.
    }
  \end{equation*}
\end{corollary}
\begin{remark}
  Note that Sreekantan does not deal with the case of cycles at
  all. If we try to argue by analogy, we should expect that in the
  case $k+1=2r$ the composed map
  $CH^r(X) \to H_{\mathcal{D}}^{2r}(X,\mathbb{Q}(r)) \to T_{r}^0$
  factors via the cycle class map and so the analogous Sreekantan
  regulator maps on homologically trivial cycles $CH^r(X)_0$ again to
  $\operatorname{coker} \left(T_{r}^{-1} \to T_{r-1}^{1}\right)$, but
  as this last map is (essentially) bijective, the corresponding
  regulator is trivial (unlike the toric regulator).
\end{remark}

\section{$K_2$ of curves and the rigid analytic regulator of P\'al}
\label{sec:pal}

In this section we assume that $X$ is of dimension $1$. We recall
from~\cite[Section~\ref{sec:globsemi}]{Bes17} the setup that will be
used here and also in some parts of Section~\ref{sec:syntomic}. By
assumption $X$ is the generic fiber of a proper $\O_K$ scheme $\XX$
with (Zariski) semi-stable reduction
\begin{equation}
  \label{eq:Yi}
  Y = \bigcup_i Y_i\;.
\end{equation}
In particular, locally near an intersection point $Y_i\cap Y_j$ there
are coordinates $x,y$ satisfying
\begin{equation}
  xy=\pi\;,\; Y_i=(x)\;,\;Y_j=(y) \label{eq:ss}
\end{equation}
(here, $(f)$ denotes the divisor of the rational function $f$).

For simplicity we will assume that components $Y_i$ and $Y_j$
intersect at at most one point.

Let $\Xgrph$ be the dual graph of $Y$ with vertices $V$ and edges $E$
(this is of course an abuse of notation as it really depends on the
particular model). The vertices correspond to the components $Y_v$
while the edges are ordered pairs of intersecting components
$(Y_v,Y_w)$ oriented from $v$ to $w$, so that an edge $e$ has tail
$e^+=v$ and head $e^-=w$. For such an edge we denote by $-e$ the same
edge with reverse orientation.

The reduction map $X\to Y$ allows us to split $X$ into rigid analytic
domains $U_v = \red^{-1} Y_v$ which are wide open spaces in the sense
of Coleman. These then intersect along annuli corresponding
bijectively to the unoriented edges of $\Xgrph$. Indeed, in terms of
the coordinates $x,y$ appearing in~\eqref{eq:ss} the annulus
corresponding to the edge $(Y_i,Y_j)$ gets mapped via $x$ (or $y$) to
the rigid analytic space $ A(|\pi|,1)$ with
\begin{equation}
  \label{eq:annulus}
  A(r,s) := \{z\in \Kbar\;,\; r<|z|<s\}\;.
\end{equation}
An orientation of an annulus fixes a sign for the residue along this
annulus and we match oriented edges with oriented annuli as
in~\cite[Definition~\ref{convention1}]{Bes17}. We use the same
notation for the edge and for the associated oriented annulus.

At this point we recall some facts about graph cohomology and harmonic
cochains on graphs. For a slightly expanded version the reader may
consult~\cite[Section~\ref{sec:globsemi}]{Bes17}. For a graph
$\Gamma=(V,E)$ and an abelian group $A$ we define $0$ and $1$ cochains
on $\Gamma$ with values in $A$ by
\begin{equation*}
  C^0(\Gamma,A) = \{f: V \to A\}\;,\;
  C^1(\Gamma,A) = \{f: E \to A\;,\; f(-e) =-f(e)\}\;.
\end{equation*}
We have a differential
\begin{equation*}
  d: C^0(\Gamma,A) \to  C^1(\Gamma,A),\; d f (e) = f(e^+) -f(e^-)\;
\end{equation*}
and the graph cohomology $H^1(\Gamma,A)$ is the cokernel of $d$. We
have a pointwise product of $c,d \in C^1(\Gamma,A)$, assuming that $A$
is a ring, defined by
\begin{equation*}
  c\cdot d = \sum_{e\in E(G)/\pm} c(e)\cdot d(e)\;,
\end{equation*}
where the sum is over unoriented edges. The kernel of the dual
differential \newcommand{\hh}{\mathcal{H}}
\begin{equation*}
  d^\ast: C^1(G,A) \to C^0(G,A),\;
  d^\ast f(v) = \sum_{e^+ = v} f(e)\;,
\end{equation*}
\newcommand{\inject}{\hookrightarrow}
\newcommand{\im}{\operatorname{Im}} 

\parindent=0cm is the space $ \hh(\Gamma,A) $ of
\emph{harmonic cochains} with values in $A$. The injection
$\hh(\Gamma,A) \inject C^1(\Gamma,A)$ induces a map
\begin{equation}
  \hh(\Gamma,A) \to  H^1(\Gamma,A).\label{eq:harmonicisog}
\end{equation}
which is an isomorphism if $A$ is a $\Q$-vector space. With $A$ a ring
again, the space $\hh(\Gamma,A)$ is orthogonal to $\im d$, hence
induces a pairing
\begin{equation}
  \label{eq:harmonicpair}
  \hh(\Gamma,A) \times  H^1(\Gamma,A) \to A\;.
\end{equation}

There is an unoriented version of the above, which will be needed for
what follows. This requires fixing an orientation of each edge. Then
we can redefine $1$-cochains with values in $A$ as functions from
unoriented edges to $A$. The differential and the dual differential
are defined as above but using the fixed orientation on each edge and
the pointwise product is unchanged. It is trivial that this
construction produced isomorphic graph cohomologies.
\begin{proposition}
  Let $X$ be as above, with dual graph $\Gamma$. We have isomorphisms
  \begin{align*}
    T_1^{-1} &\isom \hh(\Gamma,\Z)\;, \\
    T_0^1 &\isom H^1(\Gamma,\Z)\;.
  \end{align*}
  With respect to these isomorphisms the monodromy map
  $N: T_1^{-1} \to T_0^1 $ corresponds to the
  map~\eqref{eq:harmonicisog} and the product
  $T_1^{-1} \times T_0^1 \to \Z$ to the paring~\eqref{eq:harmonicpair}
  induced by the pointwise product.
\end{proposition}
\begin{proof}
  We have
  \begin{equation}
    \begin{split}
      \ctt{1}{-1}&= \bigoplus_{k\ge 0} \ct{1}{-1}{k} =  \bigoplus_{k\ge 0} \Cy{-k}{2k+2} = \Cy{0}{2}\;, \\
      \ctt{1}{0}&= \bigoplus_{k\ge 0} \ct{1}{0}{k} =  \bigoplus_{k\ge 0} \Cy{1-k}{2k+1} = \Cy{1}{1}\;, \\
      \ctt{1}{-2}&= \bigoplus_{k\ge 0} \ct{1}{-2}{k} =  \bigoplus_{k\ge 0} \Cy{-1-k}{2k+3} = 0\;, \\
      \ctt{0}{1}&= \bigoplus_{k\ge 1} \ct{0}{1}{k} =  \bigoplus_{k\ge 1} \Cy{1-k}{2k} = \Cy{0}{2}\;, \\
      \ctt{0}{2}&= \bigoplus_{k\ge 2} \ct{0}{2}{k} =  \bigoplus_{k\ge 2} \Cy{2-k}{2k-1} = 0\;, \\
      \ctt{0}{0}&= \bigoplus_{k\ge 0} \ct{0}{0}{k} = \bigoplus_{k\ge
        0} \Cy{-k}{2k+1} = \Cy{0}{1}\;.
    \end{split}\label{Tisom}
  \end{equation}
  As all the components of $\Ymm{1}$ are projective lines, their
  $CH^1$'s, as well as the $CH^0$ of the components of $\Ymm{2}$ are
  isomorphic to $\Z$ via the degree map, and the differential
  $\ctt{1}{-1} \to \ctt{1}{0}$ is the alternating sum of pushforward
  maps, which are clearly just the identity map on the appropriate
  $\Z$ summands. The chosen numbering of the components gives an
  orientation on all the edges. The unoriented edge $\{i,j\}$ gets the
  orientation $(i,j)$ with $i<j$. With this the isomorphisms~\eqref{Tisom} are clear in the unoriented version. Similarly, The
  $CH^0$ of each component of $\Ymm{1}$ is isomorphic to $\Z$ and the
  map $\ctt{0}{0} \to \ctt{0}{1} $ is an alternating sum of pullbacks,
  which are again the identities on the corresponding $\Z$ components,
  giving the identification of the monodromy operator, again in the
  unoriented version. The identification of the pairing is clear.
\end{proof}
\begin{corollary}
  We have $\hT(X,2,2)\isom \hh(\Gamma,K^\times) $ and the toric
  regulator in this case is therefore a map
  \begin{equation}
    \label{eq:toric2}
    \regt: \hmot(X,2,2) \to \hh(\Gamma,K^\times)\;.
  \end{equation}
\end{corollary}

We now recall the (somewhat reformulated) definition of the P\'al
rigid analytic regulator~\cite{Pal10}. We start by working over $\Cp$
and with closed annuli $A[r,s]$ instead of open
ones~\eqref{eq:annulus}.
\begin{definition}
  View the annulus $e=A[r,s]$ as embedded in $\PP^1$ in the obvious
  way, and let $D=D(r)$ be the disc $\{|z|< r\}$ inside $\PP^1$. For
  rational functions $f,g$ on $\PP^1$ which have no poles or zeros on
  $e$, set
  \begin{equation*}
    t_e(f,g) = \prod_{x\in D} t_x(f,g)\in \Cp^\times \;,
  \end{equation*}
  where $t_x$ is the tame symbol at the point $x$.  Let $f,g$ be
  invertible rigid analytic functions on $A[r,s]$. Let $f_n$ and $g_n$
  be sequences of rational functions on $\PP^1$ that converge to $f$
  and $g$ respectively on $A[r,s]$. Then set
  \begin{equation*}
    t_e(f,g) = \lim_{n\to \infty} t_e(f_n,g_n)\in \Cp\;.
  \end{equation*}
  For an open annulus $e=A(r,s)$ define $t_e$ to be $t_{e'}$ for any
  smaller closed annulus.  Finally, for an oriented (open) annulus $e$
  and rigid analytic functions $f,g$ on $e$, define $t_e(f,g)$ by
  choosing an orientation preserving identification of $e$ with
  $A(r,s)$.
\end{definition}
\begin{theorem}\label{palthm}
  The quantity $t_e(f,g)$ is well defined and in
  $\Cp^\times$. Furthermore we have the following:
  \begin{enumerate}
  \item $t_{-e}(f,g) = t_e(f,g)^{-1}$.\label{cond1}
  \item $t_e(f,1-f) = 1$\label{cond2}
  \item Let $U$ be the complement in $\PP^1(\Cp)$ of the union of a
    finite number of disjoint closed balls $D[r_i]$ and let $e_i$ be
    annuli $A(r_i,r_i+\varepsilon)$ where $\varepsilon$ is chosen
    sufficiently small so that the $e_i$'s are disjoint. Let $f,g$ be
    invertible meromorphic functions on $U$ which are invertible on
    the $e_i$. Then the following residue theorem holds:
    \begin{equation*}
      \prod_{x\in U} t_x(f,g) \cdot  \prod_i t_{e_i}(f,g) = 1 \;.
    \end{equation*}
  \end{enumerate}
\end{theorem}
\begin{proof}
  For the closed disc $A[r,s]$ what we defined here is what P\'al
  defines, in the course of proving~\cite[Theorem~2.2]{Pal10}, as $t_D$
  for the boundary component $D(r) $ of the connected rational
  subdomain $A[r,s]$. In particular, it is well defined. Lemma~3.4
  in~\cite{Pal10} shows that the choice of a closed annulus inside an
  open annulus does not matter and P\'al also
  shows~\cite[Theorem~3.11]{Pal10} that the construction commutes with
  morphisms of domains, which shows that it is independent of the
  choice of parameterization. Thus, the index is well defined.
  Switching the orientation corresponds to using $D=\{s<|z|\} $
  and~\cite[Theorem~3.2 (iii)]{Pal10} immediately gives~\eqref{cond1}
  (P\'al works here already with elements of $K_2$ but the result
  certainly specializes to what we have here). The formula in
  Theorem~2.2 (iv) there immediately implies~\eqref{cond2} and the
  residue Theorem is easily deduced from (iii) of Theorem~3.2.
\end{proof}
\begin{lemma}\label{paldeg}
  Let $\deg_e: \O(e)^\times \to \Z$ be defined by
  $\deg_e f = \res_e \dlog f$. Then, for a constant $c\in \Cp^\times$,
  $t_e(c,f) = c^{\deg_e(f)}$.
\end{lemma}
\begin{proof}
  This holds for rational functions $f$ with a degree function
  equaling the number of zeros in the disc $D(r)$ minus the number of
  poles, which is clearly the same as our degree, by~\cite{Pal10}. It
  then follows in general by continuity.
\end{proof}
To define the P\'al regulator we first consider functions
$f,g\in \Cp(X)^\times$ having no poles or zeros on any $e\in E$. We
can then define $\reg_P(\{f,g\})\in C^1(\Gamma,\Cp^\times) $ by
\begin{equation*}
  \reg_P(\{f,g\})(e)= t_e(f,g)\;.
\end{equation*}
Next, if $\alpha=\sum \{f_i,g_i\}$ is a formal combination of symbols
and all the functions $f_i,g_i$ are invertible on all $e\in E$, and
all its tame symbols are $1$, then
\begin{equation*}
  \reg_p(\alpha) = \prod \reg_P(\{f_i,g_i\}) 
\end{equation*}
is defined and the residue theorem implies that it is in
$\hh(\Gamma,\Cp) $. Suppose now that all the $f_i,g_i$ are in
$K(X)^\times$ and all tame symbols are $1$, but without assuming they
are invertible on the $e$'s. By making a finite field extension we can
make sure that all points where any of these functions have values in
$\{0,1,\infty\}$ are defined over $K$ and, maintaining a semi-stable
model by possibly blowing up, we get for the new graph $\Gamma'$
that none of these points is in any of $e\in E'$ and we obtain
$ \reg_p(\alpha)\in \hh(\Gamma',\Cp^\times)$. Now, \eqref{cond2} of
Theorem~\ref{palthm} shows that $\reg_P$ factors via $K_2(K(X))$. Note
that $\Gamma'$ is obtained from $\Gamma$ by subdividing edges and it
is easily seen that
$ \hh(\Gamma',\Cp^\times)\isom \hh(\Gamma,\Cp^\times)$. Finally,
\cite[Proposition~4.2]{Pal10} shows that, for functions in $K(X)$,
$\reg_P$ in fact takes values in some finite extension and compatibility with
Galois actions shows that in fact takes values in $K^\times$. To
summarize, we get a map
\begin{equation*}
  \reg_P:K_2(X)\to \hh(\Gamma,K^\times)\;.
\end{equation*}
\begin{theorem}
  The map $\reg_P$ equals the map $\regt$ from~\eqref{eq:toric2}%
  except possibly at the prime $p$. In other words, $\reg_P$ gives
  $\regtl$ upon completion at $l$ for all primes $l\ne p$
  (conjecturally also for $l=p$).
\end{theorem}
\begin{proof}
  The proof is inspired by the work of Asakura~\cite[page
  279]{Asak06}. The toric regulator completed at a prime $l$ is the
  map
  \begin{equation*}
    \hmot(X,2,2) \to H^1(K,\het^1(X\otimes_K \Kbar,\Zl(2))) \to  H^1(K,T_1^{-1}\otimes \Zl(1)) \to T_1^{-1} \otimes K^{\times(l)} 
  \end{equation*}
  and so we first need to understand the map
  $\het^1(X\otimes_K \Kbar,\Zl(1)) \to T_1^{-1}\otimes \Zl $ from
  which the second left most map above is derived by twisting once and
  taking Galois cohomology.
  \begin{proposition}\label{jacmap}
    Suppose $l\ne p$.  Let $J$ be the Jacobian of $X$.  The map
    $$\het^1(X\otimes_K \Kbar,\Zl(1)) \isom T_l(J) \to T_1^{-1}\otimes
    \Zl $$ is the limit of the maps
    \begin{equation*}
      J[l^n] \to C^1(\Gamma,\Z/l^n)
    \end{equation*}
    defined as follows: let $[D]\in J[l^n]$ by the class of a divisor
    $D$. Then $l^n D$ is the divisor of a function $f$ and the map
    sends $[D]$ to $(e\mapsto \res_e(\dlog(f)) \pmod{l^n}) $.
  \end{proposition}
  \begin{proof}
    We begin by observing that in the map described here, we could just
    as well replace $[D]$ by a torsion class in
    $\operatorname{Pic}(X/\Kbar)$, represented by some \v{C}ech
    cocycle, which is then the boundary of $(g_K)$ and map to the
    residue of $\dlog g_k$ on an annulus, as this is independent of
    $k$ modulo $n$.

    The map we need is
    \begin{equation*}
      \het^1(X\otimes_K \Kbar,\Z/l^n(1)) = \het^1(Y, \mathbb{R}\Psi \Z/l^n(1)) \to \het^1(Y, \mathbb{R}^1\Psi \Z/l^n(1)[1])\;,
    \end{equation*}
    with $\mathbb{R}^1\Psi$ the functor of ``nearby cycles'', followed
    by the identification \newcommand{\Ybar}{\bar{Y}}
    \begin{equation*}
      \mathbb{R}^1\Psi \Z/l^n(1) \isom a_{2\ast}  \Z/l^n \;,
    \end{equation*}
    where, for any $k$, $a_k: \Ymm{k} \inject \Ybar$ is the obvious
    injection, and so we attempt to understand this last
    identification following~\cite[Proposition~1.2]{Sait03}. We have
    morphisms $X\xrightarrow{j} \mathcal{X}\xleftarrow{i} Y $. Then,
    according to Saito, we have isomorphisms,
    \begin{equation}\label{saitiso}
      \theta': a_{k\ast} \Z/l^n \xrightarrow{\sim} i^\ast \mathbb{R}^k j_\ast  \Z/l^n(k)  
    \end{equation}
    (the reader would notice a shift in notation due to our definition
    of $\Ymm{k}$), and a short exact sequence
    \begin{equation}\label{saitshort}
      0\to \mathbb{R}^k\Psi  \Z/l^n \xrightarrow{\bar{\theta}} i^\ast \mathbb{R}^{k+1} j_\ast  \Z/l^n (1) \to  \mathbb{R}^{k+1}\Psi  \Z/l^n(1) \to 0\;,
    \end{equation}
    whose definition we will recall in a moment. To make sense out of
    these formulas, note that we identify sheaves on $Y$ with sheaves
    on $\Ybar$ with an action of the Galois group. Since there is no
    $\Ymm{3}$, the isomorphism above gives
    $ i^\ast \mathbb{R}^3 j_\ast \Z/l^n = 0 $ and the short exact
    sequence for $k=2$ gives $ \mathbb{R}^2\Psi \Z/l^n = 0 $. Then for
    $k=1$ the short exact sequence yields an isomorphism
    \begin{equation*}
      \mathbb{R}^1\Psi  \Z/l^n(1) \xrightarrow{\bar{\theta}} i^\ast \mathbb{R}^2 j_\ast  \Z/l^n (2)
    \end{equation*}
    from which we get the required identification by composing with
    the inverse of $\theta'$.

    Next we recall how the thetas are defined, focusing on a
    neighborhood of a point in $Y^{(2)}$ locally given by an equation
    $xy=\pi$. The special fiber $Y$ is defined there by $\pi$ and the
    two components of $Y$ passing through the point are given by the
    additional equation $x=0$ and $y=0$. Let $i_i$ be the embedding of
    $Y_i$ into $X$, $j_i: X-Y_i \to X$ the obvious open immersion, and
    $a^i: Y_i \to Y$. Saito defines a map
    \begin{equation*}
      \theta_i: a_\ast^i   \Z/l^n \to i_i^\ast \mathbb{R}^1 j_{i\ast}  \Z/l^n(1) 
    \end{equation*}
    to be the map sending $1$ to the Kummer image of the local
    generator of $\O(-Y_i)$. On our local neightborhood, these are
    clearly the Kummer images of $x$ and $y$ respectively, which we
    will denote by $(x)$ and $(y)$. Then, following Saito, the map
    \begin{equation*}
      \theta' = \sum \theta_i:  a_{1\ast}  \Z/l^n \to  i^\ast \mathbb{R}^1 j_{\ast}  \Z/l^n(1)
    \end{equation*}
    is an isomorphism and the isomorphisms~\eqref{saitiso} are deduced
    from it by taking cup products. Finally, the map
    \begin{equation*}
      i^\ast \mathbb{R}^k j_{\ast}  \Z/l^n \to  \mathbb{R}^k \Psi  \Z/l^n
    \end{equation*}
    is surjective and the map $\bar{\theta}$ is induced by the map
    $ i^\ast \mathbb{R}^k j_{\ast} \Z/l^n \to i^\ast \mathbb{R}^{k+1}
    j_{\ast} \Z/l^n(1) $ obtained by cup product with the Kummer class
    of $\pi$, which, on our neighborhood, is $(x)+(y)$. We finally
    obtain the map we want, locally on our chosen neighborhood as the
    composition
    \begin{equation}\label{saitcomp}
      \mathbb{R}^1 \Psi  \Z/l^n(1) \xrightarrow{\text{lift}}   i^\ast \mathbb{R}^1 j_{\ast}  \Z/l^n (1) \xrightarrow{((x)+(y)) \cup}   i^\ast \mathbb{R}^2 j_{\ast}  \Z/l^n (2) \xleftarrow{(x)\cup(y) \cup}  a_{2\ast}  \Z/l^n 
    \end{equation}
    where the last map is an isomorphism that needs to be
    inverted. From this it is clear that the Kummer images of $x$ and
    $y$ go to $\pm 1$ respectively. Rigidifying and taking the residue
    of the corresponding $\dlog$'s on the resulting annulus would
    obviously give the same. So, finally, our required map would map a
    torsion class in the picard group to the cocycle $(g_k)$ and then
    apply~\eqref{saitcomp} to the Kummer image of $g_k$. This is now
    clearly the same as the map claimed in the Proposition.
  \end{proof}
  Switching to $K_2$ for convenience, we can now concentrate on a
  single annulus $e$ and compute the map
  \newcommand{\Pic}{\operatorname{Pic}}
  \newcommand{\spec}{\operatorname{Spec}}
  \begin{equation}\label{wayone}
    K_2(X)/l^n \to H^1(K,\Pic(X)[l^n]\otimes \mu_{l^n})  \xrightarrow{\res_e} H^1(K,\mu_{l^n}) \to K^\times/(K^\times)^{l^n} 
  \end{equation}
  where $\mu$ denotes roots of unity and $\res_e$ is the ``$e$
  component'' of the map in Proposition~\ref{jacmap}. Let
  $A=\O(e)$. Furthermore, identify $e$ with some
  $A(r,s)\subset \PP^1$. Let $B$ be the ring of rational functions on
  $\PP^1$ regular on $e$ and let $D=D(r)$.  We have a map
  $\spec(A) \to X$ and using the pullback via this map we can write
  the map~\eqref{wayone} as the composition of
  $ K_2(X)/l^n \to K_2(A)/l^n $ with
  \begin{equation}\label{waytwo}
    K_2(A)/l^n \to H^1(K,\Pic(A)[l^n]\otimes \mu_{l^n})  \xrightarrow{\res_e} H^1(K,\mu_{l^n}) \to K^\times/(K^\times)^{l^n} \;.
  \end{equation}
  For $f\in B$ we have
  \begin{equation*}
    \res_e \dlog(f) = \sum_{x\in D} \res_x \dlog(f)
  \end{equation*}
  and therefore the restriction of~\eqref{waytwo} to $ K_2(B)/l^n $ is
  the sum over $x\in D$ of the maps
  \begin{equation*}
    K_2(B)/l^n \to H^1(K,\Pic(B)[l^n]\otimes \mu_{l^n})  \xrightarrow{\res_x} H^1(K,\mu_{l^n}) \to K^\times/(K^\times)^{l^n} \;.
  \end{equation*}
  According to~\cite[Lemma~4.3 (2)]{Asak06} for each $x\in D$ this
  last map is just the tame symbol at $x$ modulo
  $(K^\times)^{l^n}$. It is easy to see that if $f,g\in A$ are
  sufficiently well approximated by $f_i,g_i\in B$, then their images
  in $K_2(A)/l^n$ become identical. The theorem follows easily.
\end{proof}
\section{The relation with the syntomic regulator}\label{sec:syntomic}

In this section we will analyze the construction of the toric
regulator at the prime $l=p$, after tensoring with $\Q$. We will see
that the logarithm of the toric regulator may be computed using the
syntomic regulator of Nekov{\'a}{\v{r}} and Nizio\l. We will then
check our description of the toric regulator for $K_2$ of curves via
the P\'al regulator using the computation of the syntomic regulator in
this case by the first named author~\cite{Bes18}.

Let $X$ be a smooth variety over a $p$-adic field $K$.  Then,
Nekov{\'a}{\v{r}} and Nizio\l\ \cite[Theorem~5.9]{Nek-Niz14} show that
the regulator map,
\begin{equation}\label{regp}
  \reg_p: \hmot(X,k+1,r)_0 \to H^1(K,H_{\et}^{k}(X\otimes_K \bar{K}, \Qp(r)))\;,
\end{equation}
factors via the subgroup
$\hst^1(K,H_{\et}^{k}(X\otimes_K \bar{K}, \Qp(r))) $, where, for a de Rham representation $V$ of $G$,
$\hst^1(V)$ is semi-stable
cohomology, which may be interpreted as the group of Yoneda extensions
in the category of potentially semi-stable representation and may be
computed in terms of the complex $\cst(V)$ of~\cite[1.19]{Nek93},
\begin{equation}\label{sscompt}
  \cst(V): \Dst(V) \xrightarrow{(\varphi-1,N,-i)}\Dst(V)\oplus \Dst(V) \oplus \DR(V)/F^0 \xrightarrow{N+1-p\varphi+0} \Dst(V)\;.
\end{equation}
Here, $\Dst$ and $\DR$ are the functors defined by Fontaine: $\Dst(V)$
is a $K_0$-vector space, where $K_0$ is the maximal unramified
extension of $\Qp$ inside $K$, equipped with a linear nilpotent
operator $N$ (called monodromy) and a semi-linear (with respect to the
unique lift of Frobenius on $K_0$ operator $\varphi$ (Frobenius),
satisfying the relation
\begin{equation*}
  N\varphi= p \varphi N\;,
\end{equation*}
and $\DR(V)$ is a filtered $K$-vector space, and there is an
isomorphism $\Dst(V)\otimes K\to \DR(V)$. We remark that the map
between $\hst$, computed in terms of the complex $\cst$, and Galois
cohomology, is a manifestation of the Bloch-Kato exponential
map~\cite[Definition~3.10]{Blo-Kat90}.

Assume now that $X$ has totally degenerate reduction. Let
$V=H_{\et}^{k}(X\otimes_K \bar{K}, \Qp(r)) $. We first note that as the reduction of $X$ is semi-stable, the representation $V$ is semi-stable. It follows~\cite[1.24 (3)]{Nek93} that
\begin{equation}
\hst^1(K,V)=H_g^1(V)\; .\label{eq:stisg}
\end{equation}
Let us compute
$\hst^1(K,V) $.  Let $D=\Dst(V)$. Then, as a $K_0$-vector space with a
Frobenius and a monodromy operator it equals, by the semi-stable
conjecture of Fontaine, proved by Tsuji~\cite{Tsu99} to Hyodo-Kato
cohomology $H^k(Y^\times/W^\times) $. It follows from~\cite{Ras-Xar07}
that we have a slope decomposition
\begin{equation*}
  D = \bigoplus_{i+j=k} T_j^{i-j} \otimes K_0(r-j)
\end{equation*}
and the monodromy operator is compatible with the one defined on the
$T$'s. For simplicity let us renumber this as follows: We have a
vector space decomposition
\begin{equation*}
  D= \bigoplus_i D^i\;,\; D^i \isom T^i\otimes K_0\;,\; T^i = T_{i+r}^{k-2r-2i}\;,
\end{equation*}
where, with respect to the rational structure provided by $T^i$, the
Frobenius $\phi$ acts by $p^i \sigma$, so that
$(D^i)^{\phi=p^i}= T^i\otimes \Q_p$. The monodromy maps are induced by
the ones defined before $N:T^i\to T^{i-1}$. Note that by
remark~\ref{inj} the map $N:T^0\to T^{-1}$ is injective after
tensoring with $\Q$.  The filtration is compatible with the filtration
on the individual terms, where
$\DR(\Qp(i))= F^{-i} \supset F^{-i-1}=0$.  We can now compute the
semi-stable cohomology of $V$.
\begin{proposition}
  We have a short exact sequence
  \begin{equation}\label{synshort}
    0\to \DR(V)/F^0 \to  \hst^1(K,V) \to (T^{-1}/N T^0)\otimes \Qp \to 0
  \end{equation}
  and a map $\hst^1(V) \to \DR(V)/(F^0+ T^0\otimes \Qp) $ such that
  the composition
  \begin{equation}
    \label{compiten}
    \DR(V)/F^0 \to  \hst^1(K,V) \to \DR(V)/(F^0+ T^0\otimes \Qp)
  \end{equation}
  is the projection. Both of these maps are functorial.
\end{proposition}
\begin{proof}
  Let us begin by computing the cohomology of the complex $\cstp(V)$,
  which is obtained from $\cst(V)$ by dropping the de Rham component,
  \begin{equation*}
    \Dst(V) \xrightarrow{(\varphi-1,N)}\Dst(V)\oplus \Dst(V)  \xrightarrow{N+1-p\varphi} \Dst(V)\;.
  \end{equation*}
  We can decompose this last complex according to slopes
  \begin{equation*}
    \cstp(V) = \bigoplus_i \csti(V)
  \end{equation*}
  with
  \begin{equation*}
    \csti(V): D^i \xrightarrow{(p^i\sigma -1,N)} D^i \oplus D^{i-1} \xrightarrow{N+ 1- p^i\sigma} D^{i-1}\;.
  \end{equation*}
  Since $p^i\sigma -1$ is bijective unless $i=0$, we immediately see
  that $H^0(\csti)=0$ unless $i=0$, and also for $i=0$ since
  $N:T^0\to T^{-1}$ is injective after tensoring with $\Q$. Consider
  next $H^1$. Suppose $(x,y)\in D^i\oplus D^{i-1}$ represents an
  element in $H^1(\csti)$. If $i\ne 0$ we may use the bijectivity of
  $p^i\sigma -1$ to assume $x=0$ and the equation on $y$ becomes
  $(p^i\sigma -1)y=0$ so $y=0$ as well. Thus, $H^1(\csti)=0$ unless
  $i=0$. In this last case we can write explicitly
  \begin{equation*}
    H^1(\csto) = \frac{\{(x,y),\;x\in D^0,\; y\in D^{-1},\; Nx= (\sigma-1)y\}}{\{((\sigma-1)z,Nz),\; z\in D^0\}}
  \end{equation*}
  and we have the following.
  \begin{lemma}
    We have an isomorphism
    $(T^{-1}/N T^0)\otimes \Qp \xrightarrow{\sim} H^1(\csto)$ given by
    $u\mapsto (0,u)$
  \end{lemma}
  \begin{proof}
    The map is clearly well defined and injective. Surjectivity
    amounts to the statement that any element in $H^1(\csto)$ has a
    representative $(0,y)$, i.e, that any representative $(x,y)$ has
    $x\in \operatorname{Im} \sigma-1$. This is true because $N$ is
    defined over $\Q$ and is injective.
  \end{proof}
  \begin{remark}\label{twoexrem}
    The above Lemma may be interpreted for an extension
    \begin{equation}
    0\to T^{-1}\otimes \Qp(1) \to V \to T^0\otimes \Qp \to 0\label{eq:synu}
  \end{equation}
 as saying that the map
    $\hst^1(K,T^{-1}\otimes \Qp(1))\to \hst^1(K,V) $ is surjective.
  \end{remark}
  We have an obvious short exact sequence of complexes
  \begin{equation*}
    0 \to \DR(V)/F^0[1] \to \cst(V) \to \cstp(V) \to 0
  \end{equation*}
  and the associated long exact sequence together with the computation
  of the cohomology of $\cstp(V)$ immediately gives the short exact
  sequence~\eqref{synshort}. To define the map
  $\hst^1(V) \to \DR(V)/(F^0+ T^0\otimes \Qp) $ start with a
  representative $(x,y,d)\in \Dst(V)\oplus \Dst(V) \oplus \DR(V)/F^0$
  and use the computation of the cohomology of $\cstp(V)$ to see that
  it is equivalent to a representative with $x=0$. The $d$ component
  of this representative is now unique up to an element of
  $T^0\otimes \Qp$ and this gives the map. The composed
  map~\eqref{compiten} is clearly the projection.
\end{proof}

Consider now the case $V=\Qp(1)$, so that $T^{-1} = \Z$ while
$T^0=0$. Clearly, in this case the map
$\hst^1(K,\Qp(1))\to \DR(V)/F^0=K $ splits the short exact
sequence~\eqref{synshort} and we have an isomorphism
$\hst(K,\Qp(1))\isom K\oplus \Qp $.  Nekov{\'a}{\v{r}} proves the
following result.
\begin{proposition}[{\cite[1.35]{Nek93}}]\label{neklog}
  The composed map
  $K^\times \xrightarrow{Kummer} \hst(K,\Qp(1))\isom K\oplus \Qp $ is
  given by $x\to (\log(x),v(x))$.
\end{proposition}
\begin{proposition}\label{missingpiece}
  For the short exact sequence~\eqref{synshort} the composed map
  \begin{equation*}
    T^0\otimes \Qp \isom \hst^0(K,T^0\otimes \Qp) \to \hst^1(K,T^{-1}\otimes \Qp(1)) \to 
    T^{-1} \otimes K^{\times(l)} \xrightarrow{v} T^{-1} \otimes \Qp
  \end{equation*}
  is just the monodromy map.
\end{proposition}
\begin{proof}
By Proposition~\ref{neklog} we can compute the map by projecting on the cohomology of the complexes $\cstp$, where the result is easy.
\end{proof}
\begin{theorem}\label{logissyn}
  The toric regulator at $p$ exists. Furthermore we have the following
  commutative diagram,  where
  $V=H_{\et}^{k}(X\otimes_K \bar{K}, \Qp(r)) $,
  \begin{equation*}
    \xymatrix{
      {\hmot(X,k+1,r)_0} \ar[r]\ar[d] & \hst^1(K,V) \ar[r] &
      \DR(V)/(F^0+ T^0\otimes \Qp) \ar[d] \\
      {\hT(X,k+1,r)} \ar[rr]^{\log} && T^{-1}\otimes K / T^0\otimes \Qp .
    }
  \end{equation*}
  Here, the vertical map on the right is a projection relative to the
  subspace $\oplus_{i\le -2} T^i\otimes K$.
\end{theorem}
\begin{proof}
  Consider the quotient $V'=V/ W_{2r-k} V$. As this does not have non-positive $D^i$'s, and as
  $\DR(V')=F^0$ (since this is true for all the Tate subquotients), we
  easily see that $\hst^1(K,V')=0$. Thus, we may again do the
  factoring, as in Section~\ref{sec:toric-regulator}, of the regulator
  into $\hst^1(K, W_{2r-k} V) $, and then project to $\hst^1(K,V'') $
  with $V''=W_{2r-k} V/ W_{2r-k-4} V $. Furthermore, the map
  $\hst^1(K,V'')\to \hst^1(K,T^0\otimes \Qp)$ is $0$ because
  $\DR(\Qp) =F^0$ and by Remark~\ref{twoexrem}. Thus, the toric
  regulator exists at $p$ as in Section~\ref{sec:toric-regulator}. The
  commutativity of the diagram in the theorem is now straightforward
  from the following commutative diagram
  \begin{equation*}
    \xymatrix{
      {\hst^1(K,T^{-1}\otimes \Qp(1))} \ar[r] \ar[d] & \hst^1(K,V'')\ar[r] \ar[d] & 0 \\
      T^{-1} \otimes K \ar[r] & \DR(V'')/(F^0 + T^0\otimes \Qp) & \\
    }
  \end{equation*}
  and the fact that the composition of the Kummer map with the
  vertical map on the left is just the log map by Nekov{\'a}{\v{r}}'s
  result~\ref{neklog}.
\end{proof}

Nekov{\'a}{\v{r}} and Nizio\l\ define the syntomic regulator,
\begin{equation}
  \label{eq:regulator}
  \reg: \hmot(X,k+1,r) \to \hsyn(k+1,X,r)\;,
\end{equation}
into syntomic cohomology groups. These groups are constructed in such a way that there
is a spectral sequence
\begin{equation} \label{eq:spectral} E_2^{p,q}=
  \hst^p(K,\het^q(X\otimes \Kbar,\Qp(r))) \Rightarrow
  \hsyn({p+q},X,r)\;.
\end{equation}
One easily deduces from this spectral sequence and the syntomic
regulator a map
\begin{equation*}
  \reg: \hmot(X,k+1,r)_0 \to \hst^1(K,H_{\et}^{k}(X\otimes_K \bar{K}, \Qp(r))\;,
\end{equation*}
which is the same as the map~\eqref{regp}. The syntomic regulator is
computed without \'etale cohomology, using a mixture of de Rham and
(log) crystalline cohomology constructions, and so is more computable,
at least in principle, using a kind of ``$p$-adic differential
geometry'' approach. Indeed, in the good reduction case the syntomic
regulator has been defined for a long time and has been computed in
several cases, primarily by the first named
author~\cite{Bes98b,Bes-deJ98,Bes-deJ02,Bes10}. Recently, some of
these results have been extended to the semi-stable reduction
case~\cite{Bes18}.

We can summarize the results and comments of this section to this
point by the motto ``The log of the toric regulator is computed from
the syntomic regulator''. To end this section we illustrate this by
revisiting the case of $K_2$ of a totally degenerate curve and showing
how the syntomic regulator is indeed the logarithm of the P\'al rigid
analytic regulator.

Let $X$ be as in Section~\ref{sec:pal} and consider $k=1$, $r=2$
again. We have $$V=H_{\et}^1(X\otimes_K \bar{K}, \Qp(2)) $$.  Recall
that in this case $T^0=0$. We have $\DR(V)\isom \hdr^1(X/K)$ and
$F^0 \DR(V)=F^2 \hdr^1(X/K)=0$. According to Theorem~\ref{logissyn}
the logarithm of the toric regulator map agrees with the composed map
\begin{equation}\label{ktwosyn}
  \reg_p: \hmot(X,2,2) \xrightarrow{\reg_{\syn}} \hst^1(K,V)  \to  \hdr^1(X/K) \to T^{-1} \otimes K
\end{equation}
We identify the vector space on the right hand side with the space
$\hh(\Gamma,K)$ of $K$-valued harmonic cochains on the dual graph. As
expected from our motto, we get the following result
\begin{theorem}
  Let $X$ be as above. Then the diagram
  \begin{equation*}
    \xymatrix{
      {\hmot(X,2,2)} \ar[r]^{\reg_P} \ar[rd]^{\reg_p} & T^{-1} \otimes K^{\times} \ar[d]^{\log} \\
      & T^{-1} \otimes K
    }
  \end{equation*}
  commutes, with $\reg_P$ the P\'al regulator and $\reg_p$ the
  $p$-adic regulator from~\eqref{ktwosyn}.
\end{theorem}
\begin{proof}
  Suppose that an element $\alpha$ of $K_2(X)$ restricts to an element
  $\sum \{f_i,g_i\}$ in $K_2$ of the function field of $X$.
  In~\cite{Bes18} the first named author proved a formula for the cup
  product with a cohomology class $[\omega]$ of the image of $\alpha$
  under
  \begin{equation*}
    \hmot(X,2,2) \xrightarrow{\reg_{\syn}} \hst^1(K,V)  \to  \hdr^1(X/K)\;.
  \end{equation*}
  One checks easily that the projection $\hst^1(K,V) \to \hdr^1(X/K) $
  defined there coincides with the one we have been using. This
  formula was valid when $[\omega]$ is in the kernel of the monodromy
  operator $N$ ($X$ can be any curve with semi-stable reduction). To
  explain the formula and to complete the proof we first analyze
  $\hdr^1(X/K)$ in a bit more detail. We have the weight decomposition
  \newcommand{\pr}{\operatorname{pr}}
  \begin{equation*}
    \hdr^1(X/K) = T^{-2}\otimes K \oplus T^{-1} \otimes K = H^1(\Gamma,K) \oplus \hh(\Gamma,K)\;.
  \end{equation*}
  With respect to this decomposition the monodromy operator vanishes
  on $H^1(\Gamma,K) $ and maps $\hh(\Gamma,K)$ on $H^1(\Gamma,K)$
  via~\eqref{eq:harmonicisog}. The cup product makes both summands
  isotropic and gives the pointwise product~\eqref{eq:harmonicpair}
  otherwise. Consequently, if $\chi$ and $\pr$ denote the projections
  on $H^1(\Gamma,K) $ and $\hh(\Gamma,K)$ respectively, and we have
  $[\omega]\in \operatorname{Ker} N = H^1(\Gamma,K)$,
  $\beta \in \hdr^1(X/K) $, then
  \begin{equation*}
    [\omega]\cup \pr \beta = \chi([\omega])\cdot  \pr(\beta)\;.
  \end{equation*}
  We can finally introduce the formula of~\cite{Bes18}. The formula
  expresses the cup product in term of expressions assigned to the
  individual symbols $\{f,g\}$. The expression that we need is the one
  in Proposition~\ref{vcomp1}, which is the same as the expression for
  the regulator by the main
  theorem~\cite[Theorem~\ref{mainthm1}]{Bes18}. This expression is
  \begin{equation*}
    \sum_v \pair{\log(f),F_\omega;\log(g)}_{{(U_v-Z)^\dagger}} - \sum_e \chi(\omega)(e)\cdot \pair{\log(f),\log(g)}_e\;.
  \end{equation*}
  Here, $Z$ is a subset containing all the singularities of $f$ and
  $g$, but in any case, when all the components of the reduction are
  projective lines, the first term vanishes because all the ``triple
  indices'' appearing in the sum vanish
  by~\cite[Proposition~8.4]{Bes-deJ02}. Thus, the projection of the
  regulator on $\hh(\Gamma,K)$ is the map obtained by sending
  $\{f,g\}$ to
  \begin{equation*}
    e\mapsto \pair{\log(f),\log(g)}_e\;.
  \end{equation*}
  Therefore, the following Lemma completes the proof.
\end{proof}
\begin{lemma}
  For rigid analytic functions on the annulus $e$ we have
  $\pair{\log(f),\log(g)}_e= \log(t_e(f,g))$.
\end{lemma}
\begin{proof}
  Suppose, after identifying $e$ with $A(r,s)$, that $f$ and $g$ are
  rational functions on $\PP^1$. By~\cite[Proposition~4.10]{Bes98b} we
  have
  \begin{equation*}
    \pair{\log(f),\log(g)}_e = \sum_{x\in D(r)} \pair{\log(f),\log(g)}_x\;.
  \end{equation*}
  At each point in $D=D(r)$ we have, essentially by definition,
  \begin{equation*}
    \pair{\log(f),\log(g)}_x = \log t_x(f,g)\;.
  \end{equation*}
  Thus, the result is true for rational $f$ and $g$ and then is true
  in general by the definition of the P\'al regulator and by
  continuity.
\end{proof}

We close this section with a conjecture, which is suggested by the
relation between the toric regulator with both the syntomic and the
Sreekantan regulator (but note that in this conjecture we do not need
to make any assumptions about the reduction, other than being
semi-stable).
\begin{conjecture}
  The composition
  \begin{equation*}
    \hmot(X,k+1,r)_0 \xrightarrow{reg} \hst^1(K,H_{\et}^{k}(X\otimes_K \bar{K}, \Qp(r)) \to H^1(\cstp(H_{\et}^{k}(X\otimes_K \bar{K}, \Qp(r)))
  \end{equation*}
  factors via the Sreekantan regulator.
\end{conjecture}

\section{A conjectural formula for $K_1$ of surfaces}
\label{sec:k1}

One nice feature of the relation with the syntomic regulator developed
in the preceding section is that just as we are able to test formulas
for the toric regulator by taking their logarithms and comparing with the syntomic
regulator, we can look at formulas for the syntomic regulator and
attempt to exponentiate them to get conjectural formulas for the toric
regulator. In this section we present a conjecture for the toric
regulator for $K_1$ of surfaces, which is suggested by the
corresponding result of the first named author in the syntomic
case~\cite{Bes10}. Unfortunately, the formula we need is the analogue
of the one in~\cite{Bes10} for the semi-stable reduction case, and it
is still conjectural. As it will further take some work to introduce
the results we will not present the motivation here and describe the
conjecture without it.

\newcommand{\GG}{\mathcal{G}} \newcommand{\HH}{\mathbb{H}}
\newcommand{\TT}{\mathbb{T}} \newcommand{\Ta}[1]{u_{#1}}
\def\quot(#1,#2){#2\backslash #1}
We first recall some facts about
Mumford curves. Let $X$ be such a curve, given as $\GG\backslash \HH$
where $\HH$ is the Drinfeld upper half plane and $\GG$ is some
Schottky group. Let $\Gamma=(V,E) $ be the corresponding dual graph,
which is the quotient $\quot(\TT,\GG)$, with $\TT$ the tree of
$\HH$. Let $l$ be a prime. The filtration on
$M=\het^1(X\otimes \Kbar,\Zl(1)) $ takes the form of a short exact
sequence
\begin{equation}\label{mumshort}
  0\to H^1(\Gamma,\Z)\otimes \Zl(1) \to M \to \hh(\Gamma,\Z) \otimes \Zl \to 0\;.
\end{equation}
The augmented monodromy map
\begin{equation*}
  \nt: \hh(\Gamma,\Z) \to H^1(\Gamma,K^\times)
\end{equation*}
has the following description.
\begin{proposition}\label{Colemans}
  Let $\alpha\in \hh(\Gamma,\Z)$. Then there exist a unique
  $\omega_\alpha\in \Omega^1(X/K)$ such that for each $e\in E$ we have
  $\res_e(\omega_\alpha)= \alpha(e)$. Furthermore, on each $U_v$ there
  exists $f_\alpha^v \in \O(U_v)^\times$ such that
  $\dlog(f_\alpha^v) = \omega_\alpha|_{U_v}$.
\end{proposition}
This appears in~\cite{Col00} without a clear indication of the proof
(Coleman shows the existence of $\omega_\alpha$ and mentions it is
locally a $\dlog$, presumably relying on the theory of theta functions
as we will do). We show how it follows from~\cite{Ger-Put80}
(or~\cite{ManDri73}). We start by defining Theta functions.
\begin{definition}[{\cite[II, (2.3) and (2.3.5)]{Ger-Put80}}]
  Let $\gamma\in \GG$. The theta function $\Ta{\gamma}$ is defined, up
  to a multiplicative constant, as follows: Pick any $x\in \HH$ and
  let $w_{\gamma,x}$ be any function on $\PP^1$ whose divisor is
  $\gamma(x)-x$. Then
  \begin{equation*}
    \Ta{\gamma}:= \prod_{\delta\in \GG} \delta^\ast w_{\gamma,x}\;.
  \end{equation*}
  This function is independent of the choice of $x$ as shown
  in~\cite[II, (2.3.4)]{Ger-Put80}
\end{definition}
\newcommand{\te}{\tilde{e}} One can normalize the function by
normalizing $w$ in an obvious way, but we will not do this.  The
function $\Ta{\gamma} $ has no zeros or poles. It has a constant
factor of automorphy, meaning that
$\mu(\delta,\gamma):= \delta^\ast \Ta{\gamma}/\Ta{\gamma}$ is
constant. As a consequence the one form
$\omega_\gamma:=\dlog(\Ta{\gamma}) $ is $\GG$ invariant and
holomorphic, descending to a holomorphic one form on $X$.

We compute the residues of $\omega_\gamma$. On $\HH$ it is clear that
for an edge $\te$ of $\TT$ the residue of $\dlog(w_{\gamma,x}) $ on
$\te$ is non-zero if and only if $\te$ sits on the path between the
vertex $v_x$, corresponding to the domain where $x$ resides, and
$\gamma(v_x)$, and it is $\pm 1$ depending on its orientation compared
with that of the path. Averaging on $\GG$, we immediately get that the
residue of $\omega_\gamma$ on an edge $e\in E(\Gamma)$ is
\begin{equation*}
  \sum_{\delta\in \GG} \res_{\delta \te} \dlog(w_{\gamma,x})\;, \text{ $\te$ any lift of $e$.}
\end{equation*}
Now we recall that as the graph $\Gamma$ is finite, the space
$\hh(\Gamma,\Z)$ of harmonic forms on $\Gamma$ can be identified with
the first homology of $\Gamma$ with coefficients in $\Z$, which is, by
definition,
\begin{equation*}
  H_1(\Gamma,\Z):= \ker (\Z[E] \xrightarrow{d} \Z[V])\;,\; d(e) = (e^+) - (e^-)\;. 
\end{equation*}
\newcommand{\Hur}{\operatorname{Hur}}
\newcommand{\ab}{\operatorname{ab}} Furthermore, the Hurewicz
isomorphism
\begin{equation}\label{hurewitz}
  \Hur:\GG^{\ab} = H_1(\GG,\Z) \to H_1(\Gamma,\Z)
\end{equation}
can be evaluated on $\gamma\in \GG$ by choosing a vertex $v$ and
pushing down the path from $v$ to $\gamma(v)$ from $\TT$ to
$\Gamma$. This immediately gives
\begin{proposition}\label{hur1}
  The harmonic cocycle $e\mapsto \res_e \Ta{\gamma} $ is
  $\Hur(\gamma)$.
\end{proposition}
\begin{proof}[{Proof of Proposition~\ref{Colemans}}]
  By~\cite[Chapter VI (4.2) Proposition ]{Ger-Put80} the functions
  $\omega_\gamma$ span $\Omega^1(X)$. The result follows easily.
\end{proof}

Because for each $e\in E$ the function $f_\alpha^{e^+}$ and
$f_\alpha^{e^-}$ $\dlog$ to the same form on $e$, their quotient is a
constant $c_\alpha(e)\in K^\times$. This is a cocycle on $\Gamma$ with
values in $K^\times$ and its image in $H^1(\Gamma,K^\times)$ is
uniquely determined by $\alpha$. By pairing with harmonic cocycles we
get a bilinear form
\begin{equation*}
  \mu': \hh(\Gamma,\Z)\times \hh(\Gamma,\Z) \to K^\times
\end{equation*}
\begin{proposition}
  With the Hurewitz isomorphism~\eqref{hurewitz} we have $\mu'=\mu$.
\end{proposition}
\begin{proof}
  Using Proposition~\ref{hur1} what we need to prove is that for the
  form $\omega=\omega_\gamma= \dlog(u_\lambda)=\dlog(u)$, the
  homomorphism $\delta\mapsto \delta^\ast u/u $ in $H^1(\GG,K^\times)$
  respresents the same cohomology class as the map
  $e\mapsto \dlog^{-1} \omega|_{U_{e^+}} - \dlog^{-1}
  \omega|_{U_{e^-}}$ in $H^1(\Gamma,K^\times)$. But both are clearly
  the image of $\omega$ under the boundary map in the short exact
  sequence
  \begin{equation*}
    1\to K^\times \to \O \xrightarrow{\dlog} \Omega \to 0
  \end{equation*}
  either as sheaves on $X$ or, taking global sections on $\HH$, as
  $\GG$-modules.
\end{proof}
Becausue $\mu$ is symmetric by~\cite[Theorem~1]{ManDri73} we get.
\begin{corollary}\label{mpsym}
  the form $\mu'$ is symmetric.
\end{corollary}
\begin{corollary}
  We have $\nt(\alpha) = c_\alpha$.
\end{corollary}
\begin{proof}
  This is because $\mu$ gives the periods for the $p$-adic
  uniformization of the Jacobian of $X$ by~\cite{ManDri73}.
\end{proof}

We will from now onward normalize the choice of $f_\alpha^v$ in such a
way that
\begin{equation}
  c_\alpha \in \hh(\Gamma,K^\times)\;.\label{eq:harnor}
\end{equation}

This may require multiplying by some integer.

Suppose now we have two Mumford curves $X_i$, $i=1,2$, with
corresponding graphs $\Gamma_i$, and we consider the surface
$X=X_1\times X_2$. We want a formula for the toric regulator
\begin{equation*}
  \regt: \hmot(X,3,2) \to \hT(X,3,2)
\end{equation*}
\newcommand{\kon}{\Theta} given on elements of the form
\begin{equation}\label{kelement}
  \kon = \sum (C_j, g_j)\;,
\end{equation}
where $C_j\subset X$ are curves and $g_j$ is a rational functions on
$C_j$ such that the divisors of the $g_j$ cancel on $X$.

Let us first compute $\hT(X,3,2) $. Putting aside the uninteresting
terms corresponding to $H^0\otimes H^2$, the main contribution to
$\het^2(X,\Zl(2)) $ is
\begin{equation*}
  M = M_1\otimes M_2\;,\; M_i = \het^1(X_i\otimes \Kbar,\Zl(1))\;.
\end{equation*}
Taking the tensor product of the short exact sequences of the
form~\eqref{mumshort} corresponding to $M_i$ we get on $M$ a $3$-step
filtration and we may consider the interesting quotient $M'$ having
the following short exact sequence,
\newcommand\aA[1]{\hh(\Gamma_{#1},\Z)}
\newcommand\BB[1]{H^1(\Gamma_{#1},\Z)}
\newcommand{\id}{\operatorname{id}} \newcommand{\Pa}{P}
\begin{equation*}
  0 \to \left(\aA{1} \otimes \BB{2} \oplus \aA{2}\otimes \BB{1}\right)\otimes \Zl(1) \to M' \to \aA{1} \otimes \aA{2} \to 0\;,
\end{equation*}
and we may apply a projection $\Pa$ on one of the summands on the
left to get an extension
\begin{equation*}
  0 \to \left(\aA{1} \otimes \BB{2}\right) \otimes \Zl(1) \to M'' \to \aA{1} \otimes \aA{2} \to 0
\end{equation*}
The associated augmented monodromy is
\begin{equation}\label{jacpiece}
  \id_{\aA{1}} \otimes \nt_2:  \aA{1} \otimes \aA{2} \to
  \aA{1} \otimes \BB{2}\otimes K^\times
\end{equation}
and we get a regulator
\begin{equation*}
  \Pa \regt:  \hmot(X,3,2) \to \Pa  \hT(X,3,2)\;.
\end{equation*}
We use the duality between graph cohomology and harmonic cocycles to
view the resulting intermediate Jacobian $\Pa \hT(X,3,2)$, which is
the cokernel of~\eqref{jacpiece}, as bilinear forms
$\BB{1}\times \aA{2}\to K^\times$ modulo those forms which are
obtained from bilinear forms $\BB{1} \times \BB{2} \to \Z$ by
composing in the second coordinate with
$\nt_2: \aA{2} \to \BB{2}\otimes K^\times$ (to be precise, we need to
compose with the dual of $\nt_2$ but this is the same by
Corollary~\ref{mpsym}).

To construct the required form out of the element~\eqref{kelement} we
are going to further assume that for each index $j$ the curve $C_j$
has semi-stable reduction, the projections $\pi_i:C_j \to X_i$ are
finite for $i=1,2$ and that they give maps of graphs between the
corresponding dual graphs. Let $\alpha\in \aA{2}$, $\beta\in
\BB{1}$. Identify $\beta$ with a harmonic representative (which may
require again multiplying by a fixed integers, and let
$(f_\alpha^v)_{v\in V_2}$ be the corresponding functions as in
Proposition~\ref{Colemans}, normalized as in~\eqref{eq:harnor}. We
pick an orientation for the edges of $\Gamma_2$. Consider a curve
$C_j$ with a rational function $g_j$ on it. The map induced by $\pi_2$
on graphs determines an orientation on the edges of
$\Gamma_{C_j}$. Furthermore, we get pullbacked function
$h^w = \pi_2^\ast f_\alpha^{\pi_2(w)}$ for each $w\in V_{C_j}
$. Define \newcommand{\regq}{\reg_?}
\begin{align*}
  \regq(\kon)(\beta,\alpha)_j &= \sum_{e\in E_{C_j}} t_e(g_j,h^{e^+})^{ \beta(\pi_1(e))}\;, \\
  \regq(\kon)(\beta,\alpha) &= \sum_j  \regq(\kon)(\beta,\alpha)_j\;.
\end{align*}
Note that the only place where the orientation of the graph enters is
when deciding on the function $h^{e^+} $.  Let us check that this
gives a well defined element of $ \Pa \hT(X,3,2) $, fixing a single
$C=C_j$ and $g=g_j$. Because of the harmonicity condition there is no
ambiguity in $\beta$. Since we also imposed harmonicity in the
construction of $f_\alpha^v$ these function, and consequently the
functions $h^w$, are defined up to a single multiplicative factor
$c$. Correspondingly, the terms $ t_e(g,h^{e^+}) $ will change by
$ t_e(g,c)= c^{-\deg_e g}$ by Lemma~\ref{paldeg}, with
$e\mapsto \deg_e (g) = \res_e \dlog g$. This last quantity belongs to
$d C^0(\Gamma_C,\Z)$ by~\cite[Lemma~2.1]{Bes-Zer13} and thus the
regulator does not change.

Now we check what happens if we change the orientation. Suppose we
change the orientation for one $e\in E_2$. That changes the
orientation for all edges in $e' \in \pi_2^{-1} e$ and for each of
these $h^{e'+}$ is replaced by $h^{e'+}/\nt_2(\alpha)(e)$. Thus, the
change is given by some quadratic form evaluated on $\beta$ and
$\nt_2(\alpha)$ as required.

\begin{conjecture}
  We have $P \regt = \regq$.
\end{conjecture}

\end{document}